\documentclass[a4paper,12pt]{article}

\usepackage{color}

\usepackage{makeidx}

\usepackage{latexsym}
\usepackage{amscd}
\usepackage{amsmath}
\usepackage{amssymb}

\usepackage{amsthm}
\usepackage{float}
\usepackage{graphicx}
\usepackage{psfrag}

\theoremstyle{plain}
\newtheorem{theorem}{Theorem}

\newtheorem{lemma}[theorem]{Lemma}

\theoremstyle{definition}

\newtheorem{step}{Step}

\newdimen\argwidth
\def\db[#1\db]{%
 \setbox0=\hbox{$#1$}\argwidth=\wd0
 \setbox0=\hbox{$\left[\box0\right]$}
  \advance\argwidth by -\wd0
 \left[\kern.3\argwidth\box0 \kern.3\argwidth\right]}

\newcommand{\vspan}{\operatorname{span}}

\newcommand{\scTor}{\mathop{{\scT}or}\nolimits}

\newcommand{\simto}{\xrightarrow{\sim}}

\newcommand{\bC}{\ensuremath{\mathbb{C}}}

\newcommand{\bP}{\ensuremath{\mathbb{P}}}

\newcommand{\scN}{\ensuremath{\mathcal{N}}}
\newcommand{\scO}{\ensuremath{\mathcal{O}}}

\newcommand{\scT}{\ensuremath{\mathcal{T}}}

\newcommand{\scOP}{\scO_{\bP^n}}
\newcommand{\scTP}{\scT_{\bP^n}}
\newcommand{\logD}{\scT_{\bP^n}(- \log D)}
\newcommand{\scND}{\scN_{D / \bP^n}}

\newcommand{\bCx}{\bC^{\times}}

\newcommand{\Der}{\operatorname{Der}}

\title{Logarithmic vector fields along smooth divisors
in projective spaces}
\author{Kazushi Ueda and Masahiko Yoshinaga}
\date{}
\begin{document}

\maketitle

\begin{abstract}
We show that
a smooth divisor in a projective space
can be reconstructed
from the isomorphism class of the sheaf of logarithmic vector fields
along it
if and only if its defining equation is
of Sebastiani--Thom type.
\end{abstract}

\section{Introduction}

Let $D$ be a smooth divisor in $\bP^n$
defined by a homogeneous polynomial $f$ of degree $k$.
We say that $f$ is {\em of Sebastiani--Thom type}
if $f$ can be represented as the sum
$$
 f(x_0, \dots, x_n) = f_1(x_0, \dots, x_l) + f_2(x_{l+1}, \dots, x_n)
$$
for a choice of a homogeneous coordinate $(x_i)_{i=0}^n$ of $\bP^n$
and some $0 \le l \le n-1$.

We study the Torelli problem for logarithmic vector fields
in the sense of Dolgachev and Kapranov \cite{Dolgachev-Kapranov}.
For a divisor $D$ in the projective space $\bP^n$,
the sheaf $\logD$ of logarithmic vector fields along $D$
is defined as the subsheaf of the tangent sheaf
$\scT_{\bP^n}$
whose section consists of vector fields
tangent to $D$.
It is the sheafification of
$$
  D_0(-\log f)
   = \{ \delta \in \Der_A \mid \delta f = 0 \},
$$
where $A$ is a homogeneous coordinate ring of $\bP^n$
and $f \in A$ is the defining polynomial of $D$.
A divisor $D$ is said to be {\em Torelli}
if the isomorphism class of $\logD$ as an $\scO_{\bP^n}$-module
determines $D$ uniquely.
The main theorem of Dolgachev and Kapranov
\cite{Dolgachev-Kapranov}
is a condition for an arrangement of
sufficiently many hyperplanes
to be Torelli.

The main result in this paper is the following:

\begin{theorem} \label{th:main}
A smooth divisor in a projective space is Torelli
if and only if its defining equation is not
of Sebastiani--Thom type.
\end{theorem}

The strategy for the proof is the following:
\begin{enumerate}
 \item
The Jacobi ideal of a smooth divisor $D$ of degree $k$
is determined
by the set of divisors $E$ of degree $k-1$
such that the dimension of $H^0(\logD(-1)|_E)$ jumps up.
 \item
A smooth divisor is determined by its Jacobi ideal
if and only if it is not of Sebastiani--Thom type.
 \item
A divisor $D$ is not Torelli
if its defining equation is of Sebastiani--Thom type.
\end{enumerate}

As a corollary of Theorem \ref{th:main},
we give another proof of the main theorem of
\cite{Ueda-Yoshinaga_SCC}
that a smooth plane cubic curve is Torelli
if and only if its $j$-invariant does not vanish.

\section{Jacobi ideals from logarithmic vector fields}

Let $D$ be a smooth divisor of degree $k$ in $\bP^n$
defined by a homogeneous polynomial $f$,
and $\logD \subset \scTP$ be the sheaf of
logarithmic vector fields along $D$.
We have an exact sequence
$$
 0 \to \logD \to \scTP \to \scND \to 0,
$$
where $\scND$ is the normal bundle
of $D$ in $\bP^n$,
and an isomorphism
$$
\begin{array}{cccc}
 d f : & \scND & \simto & \scO_D(k) \\
  & \rotatebox{90}{$\in$} & & \rotatebox{90}{$\in$} \\
  & X & \mapsto & X f \\
\end{array}
$$
of $\scOP$-modules.
By the Euler sequence
$$
 0 \to \scOP(-1) \to \scOP^{\oplus (n+1)} \to \scTP(-1) \to 0,
$$
the space $H^0(\scTP(-1))$ of global sections of $\scTP(-1)$
is spanned by
$
 \{ \partial / \partial x_i \}_{i=0}^n.
$
The image of the map
$$
 H^0(\scTP(-1)) \to H^0(\scO_D(k-1))
$$
induced by the composition
$$
 \scTP(-1) \to \scND(-1) \to \scO_D(k-1)
$$
is the restriction to $D$ of the degree $k-1$ part
$$
 J(f)_{k-1} = \vspan \{ \partial f / \partial x_i \}_{i=0}^n
$$
of the Jacobi ideal $J(f)$ of $f$.

\begin{lemma}
For a divisor $E$ of $\bP^n$ of degree $k-1$,
the dimension of \\
$H^0(\logD(-1)|_E)$ jumps up
if and only if the defining equation of $E$
is contained in the Jacobi ideal of $D$.
\end{lemma}

\begin{proof}

Since $D$ is smooth,
$$
 \scTor^{\scOP}_1(\scO_D, \scO_E) = 0
$$
and we have an exact sequence
$$
 0 \to \logD(-1)|_E
   \to \scTP(-1)|_E
   \to \scO_{D \cap E}(k-1)
   \to 0,
$$
from which follows the long exact sequence
\begin{align*}
 0 & \to H^0(\logD(-1)|_E)
  \to H^0(\scTP(-1)|_E)
  \to H^0(\scO_{D \cap E}(k-1)) \\
   & \to H^1(\logD(-1)|_E) \to \cdots.
\end{align*}
Note that the image of the map
$$
 H^0(\scTP(-1)|_E) \to H^0(\scO_{D \cap E}(k-1))
$$
is the restriction to $D \cap E$
of the degree $k-1$ part
of the Jacobi ideal of $D$.
Since the dimension of
$H^0(\scTP(-1)|_E)$ does not depend on $E$,
the dimension of $H^0(\logD(-1)|_E)$ jumps up
if and only if the defining equation of $E$ is contained
in the Jacobi ideal of $D$.
\end{proof}

\section{Divisors from their Jacobi ideals}

We prove the following in this section:

\begin{lemma}
If two smooth distinct divisors in $\bP^n$
have identical Jacobi ideals,
their defining equations are of Sebastiani--Thom type.
\end{lemma}

\begin{proof}
We divide the proof into steps.
Let $f$ and $g$ be homogeneous polynomials
of degree $k$
defining distinct smooth hypersurfaces
such that
their Jacobi ideals 
$
 J(f) 
$
and $J(g)$ are identical.

\begin{step}
The pencil over $f$ and $g$ contains a polynomial $F$
such that
$
 \partial_0 F = \dots = \partial_l F = 0
$
and
$\{ \partial_i F \}_{i=l+1}^n$ is linearly independent
for some integer $l$ and
a suitable choice of a homogeneous coordinate
$(x_i)_{i=0}^n$ of $\bP^n$.
\end{step}

Indeed,
any pencil of projective hypersurfaces
contains a singular element $F$,
and the assumption $J(f) = J(g)$ implies that
$\partial_0 F, \dots, \partial_n F$
are linearly dependent.
Let $l$ be $n$ minus the dimension
of the linear span of $\{ \partial_i F \}_{i=0}^n$.
Then we can choose a homogeneous coordinate
so that
$
 \partial_i F = 0
$
for $i=0, \dots, l$ and
$\{ \partial_i F \}_{i=l+1}^n$ is linearly independent.
Note that one has $l < n$ since the divisors
defined by $f$ and $g$ are distinct.

\begin{step}
There exists a matrix $(a_{ij})_{i,j=0}^n$
such that $\det (a_{i, j})_{i,j=l+1}^{n} \ne 0$ and
$$
 \frac{\partial F}{\partial x_i}
  = \sum_{j=0}^n a_{ij} \frac{\partial f}{\partial x_j}
$$
for $i = 0, \dots, n$.
\end{step}

The existence of the matrix $(a_{ij})_{i,j=0}^n$
follows from the inclusion $J(F) \subset J(f)$.
We will show that
if $\det (a_{i, j})_{i,j=l+1}^{n}$ vanishes,
then the hypersurface defined by $f$ is singular.
Indeed,
the vanishing of $\det (a_{i, j})_{i,j=l+1}^{n}$
and linear independence
of $\{ \partial_i F \}_{i=l+1}^n$
and of $\{ \partial_i f \}_{i=0}^n$
imply that some linear combination of
$\{ \partial_i f \}_{i=0}^l$
is a linear combination of
$\{ \partial_i F \}_{i=l+1}^n$.
Then one can choose a homogeneous coordinate
so that $\partial_0 f$ is a linear combination of
$\{ \partial_j F \}_{j=l+1}^n$.
Assume that $\deg f \ge 2$,
since any linear form is of Sebastiani--Thom type.
Note that $F$ does not depend on $\{ x_i \}_{i=0}^l$
since $\partial_i F = 0$ for $i = 0, \dots, l$.
It follows that $[1:0:\dots:0] \in \bP^n$ is a singular point of
the hypersurface defined by $f$,
since $f(x_0, \dots, x_n)$ is the sum of
$x_0$ times some linear combination of
$\{ \partial_j F \}_{j=l+1}^n$
and terms which are zero
at $x_1 = \dots = x_n = 0$.

\begin{step}
There is a homogeneous coordinate $(X_i)_{i=0}^n$ such that
$\partial_i F = 0$ for $i = 0, \dots, l$ and
$\partial_i f \in J(F)$ for $i = l+1, \dots, n$.
\end{step}

Since the $(n-l) \times (n-l)$ matrix
$(a_{ij})_{i,j=l+1}^n$ is invertible,
one can find an $(n-l) \times (n+1)$ matrix
$(b_{ij})$
such that
$$
 \sum_{j=0}^l b_{ij} \frac{\partial f}{\partial x_j}
   + \frac{\partial f}{\partial x_i}
  \in J(F)
$$
for $i = l+1, \dots, n$.
Now make the projective coordinate transformation
from $(x_i)_{i=0}^n$ to $(X_i)_{i=0}^n$
defined by
$$
 x_j =
  \begin{cases}
    X_j + \sum_{i=l+1}^n b_{i j} X_i & 0 \le j \le l, \\
    X_j & l+1 \le j \le n.  
  \end{cases}
$$
Then one has
\begin{align*}
 \frac{\partial f}{\partial X_i}
  &= \sum_{j=0}^n
      \frac{\partial x_j}{\partial X_i}
      \frac{\partial f}{\partial x_j} \\
  &=  \begin{cases}
       \displaystyle{
        \frac{\partial f}{\partial x_i}
       } & i = 0, \dots, l, \\
       \displaystyle{
        \sum_{j=0}^n b_{ij}
          \frac{\partial f}{\partial x_j}
         + \frac{\partial f}{\partial x_i}
       } & i = l + 1, \dots, n.
      \end{cases}
\end{align*}
This implies
$$
 \frac{\partial F}{\partial X_i}
  = \frac{\partial F}{\partial x_i}
  = 0
$$
for $0 \le i \le l$
and
$$
 \frac{\partial f}{\partial X_i}
  \in J(F)
$$
for $l+1 \le i \le n$.

\begin{step}
$f$ is of Sebastiani--Thom type.
\end{step}

The fact
$$
 \frac{\partial F}{\partial X_0}
  = \dots
  = \frac{\partial F}{\partial X_l}
  = 0
$$
and
$$
 \frac{\partial f}{\partial X_i}
  \in J(F)
$$
for $l+1 \le i \le n$
shows
$$
 \frac{\partial^2 f}{\partial X_i \partial X_j} = 0
$$
for $0 \le i \le l$ and $ l+1 \le j \le n$.
This implies that $f$ is of Sebastiani--Thom type.

\end{proof}

Since the isomorphism class
of the sheaf of logarithmic vector fields
along the divisor defined by
$
 \mu F(x_0, \dots, x_l) + \nu G(x_{l+1}, \dots, x_n)
$
does not depend on the choice of $(\mu, \nu) \in (\bCx)^2$,
a divisor is not Torelli
if its defining equation is of Sebastiani--Thom type.

\section{Smooth plane cubic curves}

Theorem \ref{th:main} immediately yields
the following:

\begin{theorem}[{\cite[Theorem 7]{Ueda-Yoshinaga_SCC}}]
A smooth plane cubic curve is Torelli
if and only if its $j$-invariant does not vanish.
\end{theorem}

\begin{proof}
Since a smooth plane cubic curve has a vanishing $j$-invariant
if and only if it is defined by the Fermat polynomial
$$
 x^3 + y^3 + z^3
$$
for a suitable choice of a homogeneous coordinate,
it suffices to show that
any cubic polynomial of the form
$$
 f(x) + g(y, z)
$$
can be brought to the Fermat polynomial
by a projective linear coordinate change,
which is obvious.
\end{proof}

{\bf Acknowledgment:}
We thank Igor Dolgachev for a stimulating lecture
in Kyoto in December 2006.
An essential part of this paper has been worked out
during the conference in Krasnoyarsk in August 2007,
and we thank the organizers for hospitality
and all the participants for creating an inspiring atmosphere.
K. U. is supported
by Grant-in-Aid for Young Scientists (No.18840029).
M. Y. is supported by JSPS Postdoctoral Fellowship
for Research Abroad.

\bibliographystyle{plain}

\noindent
Kazushi Ueda

Department of Mathematics,
Graduate School of Science,
Osaka University,
Machikaneyama 1-1,
Toyonaka,
Osaka,
560-0043,
Japan.

{\em e-mail address}\ : \  kazushi@math.sci.osaka-u.ac.jp

\ \\

\noindent
Masahiko Yoshinaga

Department of Mathematics,
Graduate School of Science,
Kobe University,
1-1, Rokkodai, Nada-ku, Kobe 657-8501, Japan

{\em e-mail address}\ : \ myoshina@math.kobe-u.ac.jp

\end{document}